\documentclass[12pt]{amsart}

\usepackage{amssymb}
\usepackage{array}
\usepackage{bm}
\usepackage{mathtools}
\usepackage[colorlinks=true,linkcolor=blue,citecolor=blue,urlcolor=blue]{hyperref}
\usepackage{fancynum}\setfnumgsym{\hskip 2pt}

\theoremstyle{theorem}
\newtheorem{theorem}{Theorem}[section]
\newtheorem*{thm*}{Theorem}

\newtheorem{proposition}[theorem]{Proposition}

\numberwithin{equation}{section}
\theoremstyle{definition}

\newcommand{\cp}{\mathbb{C}P}
\newcommand{\zz}{\mathbb{Z}}

\begin{document}

\title{Moduli spaces of $6$ and $7$-dimensional complete intersections}
\author{Jianbo Wang}
\address{Department of Mathematics, School of Science, Tianjin University\newline
\indent Weijin Road 92, Nankai District, Tianjin  300072, P.R.China}
\email{wjianbo@tju.edu.cn}

\thanks{The author is supported by NSFC grant No.11001195 and Beiyang Elite Scholar Program of Tianjin University, No.60301016.}
\begin{abstract}
This paper proves the existence of homeomorphic (diffeomorphic) complex $6$-dimensional ($7$-dim) complete intersections that belong to components of the moduli space of different dimensions. These results are given as a supplement to earlier result on $5$-dimensional complete intersections.
\end{abstract}

\maketitle

\section{Introduction}

Let $X_n(\underline{d})\subset \cp^{n+r}$ be a smooth complete
intersection of  multidegree $\underline{d}=(d_1,\dots,d_r)$. The product $d_1d_2\cdots d_r$
is called the total degree, denoted by $d$.

Libgober and Wood (\cite{LW86}) show the existence of homeomorphic complete intersections of dimension $2$ and diffeomorphic ones of dimension $3$ which belong to components of the moduli space having different dimensions. P. Br{\"u}ckmann (\cite{B,B00}) shows that there are families of arbitrary length $k$ of complete intersections in $\cp^{4k-2}$ and $\cp^{6k-2}$ (resp. $\cp^{5k-2}$) consisting of homeomorphic complete intersections of dimension $2$ and $4$ (resp. diffeomorphic ones of dimension $3$) but that belong to components of the moduli space of different dimensions. By the help of Theorem 1.1 in \cite{FW}, the author generalized the complex dimension of complete intersections to dimension $5$ (\cite{Wang}).

The goal of this paper is to give the following theorem, which is a supplement to the results of Br{\"u}ckmann \cite{B,B00} and the author \cite{Wang}.
\begin{theorem}
For each integer $k>1$, there exist $k$ homeomorphic (diffeomorphic) complex $6$-dimensional $(7-dim)$ complete intersections in $\cp^{8k-2}~  (\cp^{15k-8})$ belonging in the moduli space to components with different dimensions.\end{theorem}

\noindent {\bf Acknowledgement.}  The author would like to thank Guo Xianqiang for his warm email conversation.

\section{Moduli spaces of complex $6$ and $7$ dimensional complete intersections}

Let $X_n(\underline{d})\subset \cp^N$, where $n\geqslant 2, ~\underline{d}=(d_1,\dots,d_r), d_i\geqslant 2$ and $r=N-n$. As usual, define the power sums
$\displaystyle s_i=\sum_{j=1}^{r}d_j^i$ for $1\leqslant i\leqslant n$, then the Pontrjagin numbers and Euler characteristic of $X_n(\underline{d})$ depend only on the dimension $n$, total degree $d$ and power sums $s_i$(\cite[$\S 7$]{LW82}).
Assume that $X_n(\underline{d})$ is not a {\bf K}$3$-surface or a quadratic hypersurface, then from \cite[Lemma 3]{B},
the explicit formula for moduli space dimension is
\begin{align}\label{dimM}
 m(\underline{d})\coloneqq  & ~m(X_n(\underline{d}))= 1-(N+1)^2+\sum_{i=1}^r\binom{N+d_i}{N} \\
& +\sum_{i=1}^r\sum_{j=1}^r(-1)^j\sum_{1\leqslant k_1<\cdots<k_j\leqslant r}\binom{N+d_i-d_{k_1}-\cdots-d_{k_{j}}}{N}.\nonumber
\end{align}
Where $\displaystyle\binom{m}{N}=0$ for $m<N(m\in \zz)$.

\begin{theorem}\label{main1}
For each integer $k>1$, there exist $k$ homeomorphic complex $6$-dimensional complete intersections in $\cp^{8k-2}$ belonging in the moduli space to components with different dimensions.
\end{theorem}

\begin{proof}
Consider the following two multidegrees(\cite{Chen}\footnote{The two multidegrees are firstly known as a solution of "Equal Products and Equal Sums of Like Powers",
by Chen Shuwen in 2001.})
\begin{equation}\label{6dim}
\begin{split}
\underline{d}=  & (2323, 2241, 2231, 2117, 2079, 1957, 1953, 1899), \\
\underline{d}^\prime=  & (2321, 2263, 2187, 2163, 2037, 2001, 1919, 1909).
\end{split}
\end{equation}
They have the same total degree $d$, power sums $s_1,\dots, s_6$ as follows:
\begin{align*}
d & =\fnum{371008634983489635445991601},\\
s_1 & =\fnum{16800}, \\
s_2 & =\fnum{35449960}, \\
s_3 & =\fnum{75160663200}, \\
s_4 & =\fnum{160103709636808}, \\
s_5 & =\fnum{342612368928228000}, \\
s_6 & =\fnum{736443048260836419880}.
\end{align*}
The corresponding complete intersections have the same Pontrjagin numbers and Euler characteristic by \cite[$\S 7$]{LW82},
so the two complete intersections
\begin{align*}
& X_6(2323, 2241, 2231, 2117, 2079, 1957, 1953, 1899), \\
& X_6(2321, 2263, 2187, 2163, 2037, 2001, 1919, 1909)
\end{align*}
are homeomorphic by \cite[Theorem 1.1]{FW}, but have different moduli space dimensions:
\begin{align*}
m(\underline{d}) & = \fnum{4639611966677972182663146217041064938},\\
m(\underline{d}^\prime) & = \fnum{4639610187986885926979324513081980800}.
\end{align*}
Using the method from \cite{LW86}, define the composed multidegree
\begin{equation}\label{ComposedMuldegree}
d_{\lambda,\mu}=(\underbrace{\underline{d},\dots,\underline{d}}_{\lambda},
\underbrace{\underline{d}^\prime,\dots,\underline{d}^\prime}_{\mu}), \lambda+\mu=s\geqslant 1 .
\end{equation}
The composed multidegrees $d_{0,s},d_{1,s-1},\dots,d_{s,0}$ have the same power sums $s_1,s_2,\dots,s_6$ respectively, so the corresponding complete intersections $X_6(d_{0,s}), X_6(d_{1,s-1}), \dots, X_6(d_{s,0})$ are homeomorphic to each other.
We can prove the following inequality:
\begin{equation*}
 m(d_{\lambda+1,\mu-1})-m(d_{\lambda,\mu})>0,  0\leqslant \lambda<s=\lambda+\mu.
\end{equation*}
This inequality will be proved in the coming Proposition.

Now, the sequence $m(d_{\lambda,s-\lambda})|_{\lambda=0,1,\dots,s-1}$ is strictly monotonously increasing. Let $k=s+1$, there exist $k$ 6-dimensional complete intersections $X_6(d_{0,s}), X_6(d_{1,s-1}), \dots, X_6(d_{s,0})$ in $\cp^{8s+6}=\cp^{8k-2}$ with the desired properties. The proof is finished.
\end{proof}
\begin{proposition}\label{Pro-mis}
\begin{equation*}
 m(d_{\lambda+1,s-\lambda-1})-m(d_{\lambda,s-\lambda})>0,  0\leqslant \lambda<s.
\end{equation*}
\end{proposition}
\begin{proof}
For the chosen multidegrees $\underline{d}$ and $\underline{d}^\prime$ in \eqref{6dim},
\begin{align} \label{dimMC}
&  m(d_{\lambda,s-\lambda})=1-(N+1)^2+\Big[\lambda\sum_{d_i\in\underline{d}}+
(s-\lambda)\sum_{d_i\in\underline{d}^\prime}\Big]\binom{N+d_i}{N}\\
& -\Big[\lambda\sum_{d_i\in\underline{d}}+(s-\lambda)\sum_{d_i\in\underline{d}^\prime}\Big]
\hskip -.1cm\sum_{\substack{1\leqslant k\leqslant 8s \\ d_{k}\in d_{\lambda,s-\lambda}}}\hskip -.1cm\binom{N+d_i-d_{k}}{N}, \nonumber
\end{align}
Where, the index $j$ in \eqref{dimM} is maximally $1$ that is determined by $\max(\underline{d},\underline{d}^\prime)=2323$ and $\min(\underline{d},\underline{d}^\prime)=1899$.  So,
\begin{align}\label{Diff6dimM}
&  m(d_{\lambda+1,s-\lambda-1})- m(d_{\lambda,s-\lambda})  \\
= &\Big[\sum_{d_i\in\underline{d}}-\sum_{d_i\in\underline{d}^\prime}\Big]\binom{N+d_i}{N}   \nonumber \\
& +\Bigg\{-\Big[(\lambda+1)\sum_{d_i\in\underline{d}}+(s-\lambda-1)\sum_{d_i\in\underline{d}^\prime}
\Big]\hskip -.2cm
\sum_{\substack{1\leqslant k\leqslant 8s  \\ d_{k}\in d_{\lambda+1,s-\lambda-1}}}  \nonumber\\
& +\Big[\lambda\sum_{d_i\in\underline{d}}+(s-\lambda)\sum_{d_i\in\underline{d}^\prime}\Big]
\hskip -.2cm
\sum_{\substack{1\leqslant k\leqslant 8s \\ d_{k}\in d_{\lambda,s-\lambda}}}\Bigg\} \binom{N+d_i-d_{k}}{N} \nonumber \\
\coloneqq & M_{0}(\lambda,s)+M_{1}(\lambda,s) , \nonumber
\end{align}
Where $M_{0}(\lambda,s), M_{1}(\lambda,s)$ are polynomials of invariants $s$ and $\lambda~ (N=8s+6)$:

\begin{align}\label{6-j=0}
M_{0}(\lambda,s)=& \Big[\sum_{d_i\in\underline{d}}-\sum_{d_i\in\underline{d}^\prime}\Big]\binom{N+d_i}{N},\\
M_{1}(\lambda,s)= &\Big[
-(\lambda+1)\sum_{d_i\in\underline{d}}\sum_{d_{k}\in d_{\lambda+1,s-\lambda-1}}-(s-\lambda-1)\sum_{d_i\in\underline{d}^\prime}\sum_{d_{k}\in d_{\lambda+1,s-\lambda-1}} \nonumber\\
& +\lambda\sum_{d_i\in\underline{d}}\sum_{d_{k}\in d_{\lambda,s-\lambda}}+(s-\lambda)\sum_{d_i\in\underline{d}^\prime}\sum_{d_{k}\in d_{\lambda,s-\lambda}}
\Big]\binom{N+d_i-d_{k}}{N} \nonumber\\
= &\Big[
(-2\lambda-1)\sum_{d_i\in\underline{d}}\sum_{d_{k}\in \underline{d}}+(1+2\lambda-s)\sum_{d_i\in\underline{d}}\sum_{d_{k}\in \underline{d}^\prime}  \label{6-j=1}\\
& +(1+2\lambda-s)\sum_{d_i\in\underline{d}^\prime}\sum_{d_{k}\in \underline{d}}+(2s-2\lambda-1)\sum_{d_i\in\underline{d}^\prime}\sum_{d_{k}\in \underline{d}^\prime}
\Big]\binom{N+d_i-d_{k}}{N}. \nonumber
\end{align}

From the above, we see that \eqref{Diff6dimM} is exactly a polynomial of $s, \lambda$ with complicated coefficients and higher degree. Using the technical computational software  {\bf\emph{Mathematica}}, \eqref{Diff6dimM} can be computed by executable program (see the code listed in \cite[$\S 5$]{Wang}).
Finally, we calculate the following results:
\begin{align*}
& m(d_{1,0})-m(d_{0,1})=\fnum{1778691086255683821703959084138},\\
& m(d_{2,0})-m(d_{1,1})=\fnum{4499576565311886952937393989311636807018493942453},\\
& m(d_{1,1})-m(d_{0,2})=\fnum{4499576565312040117972354794044912596557706541183}.
\end{align*}
More generally,
\begin{equation*}
 m(d_{\lambda+1,s-\lambda-1})-m(d_{\lambda,s-\lambda})>10^{65}, 0\leqslant \lambda <s, s\geqslant 3.
\end{equation*}

Thus, it is clear that, with any fixed $s\geqslant 1, s>\lambda\geqslant 0$, $m(d_{\lambda,s-\lambda})$ form a strictly monotonously increasing sequence for $\lambda$.
\end{proof}

\begin{theorem}\label{main2}
For each integer $k>1$, there exist $k$ diffeomorhic complex $7$-dimensional complete intersections in $\cp^{15k-8}$ belonging in the moduli space to components with different dimensions.
\end{theorem}

\begin{proof}
Consider the following two multidegrees(\cite{Guo}\footnote{These multidegrees are founded by Guo Xianqiang using artificial hand calculation.})
\begin{equation}\label{7dim}
\begin{split}
\underline{d} & =(608,592,572,516,500,453,450,424,423,408,396,366,339,312,309), \\
\underline{d}^\prime & =(604,600,564,528,488,456,452,429,416,412,387,375,333,318,306).
\end{split}
\end{equation}
They have the same total degree $d$, powers sums $s_1,\dots,s_{7}$.
\begin{align*}
d & = \fnum{3753247176539885786786848165802803200000},\\
s_1 & =\fnum{6668},\\
s_2 & =\fnum{3094964},\\
s_3 & =\fnum{1495641932},\\
s_4 & =\fnum{749415139508},\\
s_5 & =\fnum{387496273524068},\\
s_6 & =\fnum{205753667680942844},\\
s_7 & =\fnum{111680899229310068732}.
\end{align*}
Then the corresponding complete intersections
\begin{align*}
& X_7(608,592,572,516,500,453,450,424,423,408,396,366,339,312,309), \\
& X_7(604,600,564,528,488,456,452,429,416,412,387,375,333,318,306)
\end{align*}
are homeomorphic by \cite[Theorem 1.1]{FW}. Furthermore
\begin{center}
$d=2^{28}\times 3^{13}\times 5^{5}\times 11^{2}\times 13^2\times 17\times 19\times 37\times 43\times 47\times 53\times 61\times 103\times 113\times 151$,
\end{center}
so they are diffeomorphic by Traving {\cite{Tr}} (see also  \cite[Theorem A]{Kr}). However, they have different moduli space dimensions:
\begin{align*}
m(\underline{d}) & = \fnum{44406795197386326965368167342722355968367},\\
m(\underline{d}^\prime) & = \fnum{44384030917398245056066270542147363962375}.
\end{align*}
Similarly, for the composed multidegrees $d_{0,s},d_{1,s-1},\dots,d_{s,0}$ (cf. \eqref{ComposedMuldegree}), the corresponding complete intersections $X_7(d_{0,s}), X_7(d_{1,s-1}), \dots, X_7(d_{s,0})$ are diffeomorphic to each other.

For the chosen multidegrees $\underline{d}$ and $\underline{d}^\prime$ in \eqref{7dim},
since $\max(\underline{d},\underline{d}^\prime)=608$ and
$\min(\underline{d},\underline{d}^\prime)=306$, then

\begin{align}\label{Diff7dimM}
&  m(d_{\lambda+1,s-\lambda-1})- m(d_{\lambda,s-\lambda})  \\
= &\Big[\sum_{d_i\in\underline{d}}-\sum_{d_i\in\underline{d}^\prime}\Big]\binom{N+d_i}{N}   \nonumber \\
& +\Bigg\{-\Big[(\lambda+1)\sum_{d_i\in\underline{d}}+(s-\lambda-1)\sum_{d_i\in\underline{d}^\prime}
\Big]\hskip -.2cm
\sum_{\substack{1\leqslant k\leqslant 15s  \\ d_{k}\in d_{\lambda+1,s-\lambda-1}}}  \nonumber\\
& +\Big[\lambda\sum_{d_i\in\underline{d}}+(s-\lambda)\sum_{d_i\in\underline{d}^\prime}\Big]
\hskip -.2cm
\sum_{\substack{1\leqslant k\leqslant 15s \\ d_{k}\in d_{\lambda,s-\lambda}}}\Bigg\} \binom{N+d_i-d_{k}}{N} \nonumber \\
\coloneqq & M_{0}(\lambda,s)+M_{1}(\lambda,s) , \nonumber
\end{align}
where $M_{0}(\lambda,s), M_{1}(\lambda,s)$ are same as \eqref{6-j=0},\eqref{6-j=1}.
Finally, we calculate the following results:
\begin{align*}
& m(d_{1,0})-m(d_{0,1})\\
= & \fnum{22764279988081909301896800574992005992},\\
& m(d_{2,0})-m(d_{1,1}) \\
= &\fnum{33455700664468562980578980033713637615501603407170478458745},\\
& m(d_{1,1})-m(d_{0,2}) \\
= & \fnum{33455700663954152769609839164207754699356499185723378335507}.
\end{align*}
More generally,
\begin{equation*}
 m(d_{\lambda+1,s-\lambda-1})-m(d_{\lambda,s-\lambda})>3\times 10^{76}, 0\leqslant \lambda <s, s\geqslant 3.
\end{equation*}

Thus, the sequence $m(d_{\lambda,s-\lambda})|_{\lambda=0,1,\dots,s-1}$ is strictly monotonously increasing. Let $k=s+1$, there exist $k$ $7$-dimensional complete intersections $X_7(d_{0,s}), X_7(d_{1,s-1}), \dots, X_7(d_{s,0})$ in $\cp^{15s+7}=\cp^{15k-8}$ with the desired properties.
\end{proof}

\end{document}